\documentclass[12pt,a4paper]{scrartcl} 
\KOMAoptions{%
draft=false,
abstract=true,
numbers=enddot
}

\usepackage{amsmath}    
\usepackage{amsthm}     
\usepackage{amsfonts,amssymb}
\usepackage{array}

\usepackage[labelsep=period,justification=centerlast]{caption}
\usepackage{indentfirst}

\usepackage{xcolor}
\usepackage{graphicx}

\usepackage[dvipdfmx]{hyperref}
\hypersetup{%
colorlinks=true, linkcolor=blue, citecolor=blue, filecolor=blue, urlcolor=blue,
pdfauthor={A. Kurmosenko}, pdftitle={Bounds for spiral splines} , pdfsubject=, pdfkeywords={spiral biarc bilens}}

\addtokomafont{section}{\Large}
\addtokomafont{subsection}{\large\itshape}

\DeclareMathOperator{\const}{const}

\newcommand{\Skip}[1]{}

\begin{document}

\setlength{\abovedisplayskip}{6pt plus 2pt minus 3pt}  
\setlength{\belowdisplayskip}{5pt plus 2pt minus 3pt}  


\newcommand{\atanh}{\mathop{\rm arth}\nolimits}

\newtheorem{Prop}{Proposition}
\newtheorem{Stat}{Statement}
\newtheorem*{thm}{Theorem}

\newtheoremstyle{Bprop}%
{3pt}
{}
{\itshape}
{}
{}
{}
{.5em}
{}

\newtheorem{Bprop}{Property}

\newsavebox{\TmpBox}
\newcommand{\tmp}{}
\newlength{\tmplength}                    %

\newcommand{\Diam}[1]{\varnothing_{#1}}
%
\newcommand{\Equa}[2]{\begin{equation}#2\label{#1}\end{equation}}
\newcommand{\equa}[1]{\[ #1 \]}

%
%
\newcommand{\myfrac}[2]{{\ifmmode{}^{#1}\!/_{\!#2}\else${}^{#1}\!/_{\!#2}$\fi}}
\newcommand{\BR}[1]{\left[#1\right]}
\newcommand{\Brace}[1]{\left\{#1\right\}}

%
%
\newcommand{\abs}[1]{\left\lvert#1\right\rvert}
\newcommand{\sgn}{\mathop{\rm sgn}\nolimits}
\newcommand{\Deg}[1]{{\ifmmode{#1}^\circ\else${#1}^\circ$\fi}}
%
\newcommand{\Arc}[1]{\displaystyle{\buildrel\,\,\frown\over{#1}}}
\newcommand{\ieq}{\,{=}\,}
\newcommand{\ineq}{\,{\not=}\,}
\newcommand{\ilt}{\,{<}\,}
\newcommand{\igt}{\,{>}\,}
\newcommand{\In}{\,{\in}\,}
\newcommand{\ile}{\,{\le}\,}
\newcommand{\ige}{\,{\ge}\,}

\newcommand{\Eqref}[1]{\stackrel{\eqref{#1}}{=}}
%
%
%
\newcommand{\GT}[1]{$#1\igt0$}
\newcommand{\GE}[1]{$#1\ige0$}
\newcommand{\LT}[1]{$#1\ilt0$}
\newcommand{\LE}[1]{$#1\ile0$}
\newcommand{\EQ}[1]{$#1\ieq0$}
\newcommand{\NE}[1]{$#1\ineq0$}
%
\newcommand{\So}{\quad\Longrightarrow\quad}
\newcommand{\Vect}[1]{\overrightarrow{#1}}
\newcommand{\nvec}[1]{\mathbf{n}(#1)}
\newcommand{\HM}{\hphantom{{-}}}

\def\figurename{Fig.}
%
%
\def\FigDir{}

\newcommand{\Figref}[1]{\ref{F#1}}
\newcommand{\RefFig}[2][]{Fig.\;\Figref{#2}#1} 
\newcommand{\RefFigs}[2][]{Figs.\;\Figref{#2}#1} 
\newcommand{\Infigw}[2]{
\includegraphics[width=#1]{\FigDir#2.eps}}

\newcommand{\Infig}[3]{\Infigw{#1}{#2}\caption{{\small #3}}\label{F#2}}


\newcommand{\Pfig}[3]{
\Infig{#1}{#2}{#3}}

\newcommand{\eg}{e.\,g.}

\newcommand{\quo}[1]{``#1''}  

\author{Alexey Kurnosenko}
\date{}

\title{Bounds for spiral and piecewise spiral splines}

\maketitle{}

\begin{abstract}
This note is the updated outline of the article 
{\em ``Interpolational properties of planar spiral curves''},
Fund. and Applied Math., {7}(2001), N2, 441--463, published in Russian.
The main result establishes boundary regions for spiral and piecewise spiral splines,
matching given data.
The width of such region can serve as the measure of fairness of the
point set, subjected to interpolation.
Application to tolerance control of curvilinear profiles is discussed.
\end{abstract}

\section{Introduction}

This note is intended to present in English the results of the article \cite{AKinterpol},
which has arosed interest among researchers, 
working with spirals, i.\,e.  planar curves with monotone curvature. 

It is not a pure translation of \cite{AKinterpol}. Some auxiliary results of the article,
such as statements~\eqref{QVogt},
were originally developed for \quo{very short} spiral arcs,
namely, one-to-one projectable onto the chord.
These statements are now essentially expanded,
and presented to readers
as the elements of the theory of spiral arcs
(see \cite{AKparab, Bilens} and references herein).
Here corresponding proofs are omitted.

Only final results of article \cite{AKinterpol} are reproduced in this note.
They establish boundary regions for spiral and piecewise spiral splines, matching given planar point set.
The width of such region is considered as the objective measure of fairness of the given data.

Section~\ref{sec:Demo} and Appendix \ref{sec:Test}, absent in the original article, 
consider the subject in the view of tolerances control of curvilinear profiles.

\section{Demonstration of the bounding region}\label{sec:Demo}


We consider problems, related to interpolation of planar point sets.
Often one can more or less definitely decide: \quo{this interpolant seems to be good, and this one
is not} (\eg, contains unwanted oscillations). 
In brief, below we consruct the region, enclosing all \quo{good interpolants}.

To begin with, let us demonstrate such regions in the context of two problems,
{\em interpolation of function} vs {\em curve interpolation}.
Let points $(x_i,y_i$) be defined by some function $y\ieq f(x)$
as in \RefFig{Function}.
For the point set {a)}, among various interpolants, we definitely reject non-monotonous ones,
because the sequence $y_i$ is monotone. We assume that the designer of this point set takes care
to present all existing extrema of the function to improve the quality of future interpolation.
The region, bounding all monotone interpolants, is shown as the union of rectangles.

\begin{figure}[t]
\centering%
\Pfig{.8\textwidth}{Function}{Trivial bounding regions for function interpolation}%
\end{figure}

\begin{figure}[b]
\centering%
\Pfig{1\textwidth}{Tolerance}{
Drawing with curvilinear profile ABCD; 
blue: bounding region, enclosing all possible spiral interpolants;
green: bounds for the same profile with 7 interpolation nodes.}%
\end{figure}
      
Now consider the sequence of non-monotone data with maximum at point~5 (point set in \RefFig[(b,c)]{Function}).
If the point $(x_5,y(x_5))$ definitely corresponds to the maximum of $y(x)$,
bounding region can still be constructed (the case~{b)}).

For {\em function interpolation} these regions are trivial and not interesting.
Analogous regions for {\em curve interpolation} turned out to be much more interesting and useful.
Drawing with curvilinear profile $ABCD$ in \RefFig{Tolerance} yields an example.
The coordinates $(x_i,y_i)$ of the interpolation nodes could result from some known curve
$[x(t),y(t)]$, which might have an exact representation in rather complicated expressions or, \eg,
in terms of differential equations. 
%
In real practice such profile would be presented by 20--50 points.
We do with only 4 points, just to be able to examine the bounding region
without a microscope.

Data for {\em curve} interpolation include values of two functions $x(t)$ and $y(t)$,
with no argument values~$t_i$. Monotonicity and extrema of these functions have no importance:
these features are not even invariant under rotations. To obtain some invariant conclusions,
one would like to analyze the behavior of curvature $k(t)$. The analogues of ``unwanted extrema'' of $f(x)$
might be extra extrema of curvature. Contrary to extrema of function, 
whose minimal number is directly visible from the plots like \RefFig{Function},
a minimum of curvature extrema is far from being evident.
Nevetherless, given data $x_i,y_i$ 
(supplemented with tangents at the endpoints) allows us to
detect if the original curve could be a spiral.

Below we prove that all possible spiral interpolants
stay within the narrow region (shown in blue in the above example $ABCD$, of the width ${\approx}1.1$).
Conversely, any curve, going beyond these limits, is not a spiral.

\section{Geometric preliminaries}

First we explain the simplest construction of the bounding region.
Denote
\equa{%
   \nvec{\varphi}=(\cos\varphi,\sin\varphi)
}
the unit vector, associated with the angle~$\varphi$.
A set of 12 points $P_1,\ldots,P_{12}$
and end tangents $\nvec{\tau_1},\,\nvec{\tau_{12}}$
in \RefFig[(a)]{Region} is the data to be interpolated.
In \RefFig[(b)]{Region} 12 circular arcs are traced as follows: 
arc  $P_1P_2$ of the curvature~$q_1$ passes through points $P_1$ and~$P_2$, 
matching at~$P_1$ given tangent~$\nvec{\tau_1}$;
circular arcs $P_1P_2P_3$,  $P_2P_3P_4$, $\ldots$, $P_{10}P_{11}P_{12}$,
with curvatures $q_2,\ldots,q_{11}$ pass
through triples of consecutive points;
arc $P_{11}P_{12}$ of the curvature~$q_{12}$ matches given tangent $\nvec{\tau_{12}}$ at the endpoint.
Thus we obtain, on each chord $P_jP_{j+1}$, two circular arcs, 
forming the lens.
Together these lenses form the sought for region.


\begin{figure}[t]
\centering%
\Pfig{.9\textwidth}{Region}{Construction of the bounding region for spiral interpolants (simple version)}%
\end{figure}

\subsection{Spline data}

For the given set of $N$ points denote $M\ieq N{-}1$ the number of chords;
denote also the length of \hbox{$j$-th} chord as $2c_j=\abs{{P_jP_{j+1}}}$,
and its angular direction as $\mu_j$:
\begin{align}
      &1\le j\le M{:}& &c_j=\dfrac12\sqrt{(x_{j+1}{-}x_j)^2+(y_{j+1}{-}y_j)^2},&&
          \cos\mu_j=\dfrac{x_{j+1}{-}x_j}{2c_j},
          \sin\mu_j=\dfrac{y_{j+1}{-}y_j}{2c_j};\nonumber\\
      &j\ieq0,\;j\ieq N{:}& &c_0 = c_N = 0,\:\mu_0=\tau_1,\:\mu_N=\tau_N.& &\label{AddPt}
\end{align}      

Definitions \eqref{AddPt} imitate boundary tangents,
as if two additional points, $P_0$ and $P_{N+1}$, were added, forming
two infinitesimal pseudo-chords, $P_0 P_1$ and $P_N P_{N{+}1}$, 
keeping specified directions.
Two-point arcs $P_1P_2$ and $P_{N-1}P_N$ can now be described by three-point formulas
\eqref{3pt},\eqref{XiEta} as $P_0P_1P_2$ and $P_{N-1}P_N P_{N+1}$.
Turning angles $\rho_j$ of the chord in every \hbox{$j$-th} node, $1\le j\le N$, 
and signed curvatures $q_j$ of directed circles $P_{j-1}P_jP_{j+1}$ can be defined by
\Equa{3pt}{%
      \setlength{\unitlength}{1pt}
      \parbox[c]{150pt}{\begin{picture}(150,35)%
         \put(0,0){\Infigw{150pt}{AuxFigRho}}
      \end{picture}}\quad
      \begin{array}{l}
        \rho_j=\mu_j{-}\mu_{j-1},\quad  q_{j}=\dfrac{\sin\rho_j}{d_{j} },\text{~~where}\\
        d_{j} =\dfrac{1}{2}\abs{{P_{j-1}P_{j+1}}}=\sqrt{c_{j-1}^2+2c_{j-1}c_j\cos\rho_j+c_j^2}\,.
    \end{array}   
}
Curvatures $q_j$ are expected to be close to unknown curvatures~$k_j$ in the nodes 
of the original curve.

\begin{figure}[t]
\centering%
\Pfig{.92\textwidth}{XiEta}{Definition of the angles $\alpha,\beta$, $\xi,\eta$ and
 illustration to the proof of Prop.\,1. Nodes 1,\,2,\,3,\,4 are those of interpolated spiral
(dotted); dashed line are circles of curvature at nodes 2,\,3.}%
\end{figure}

\RefFig[(a)]{XiEta} illustrates unknown tangents $\nvec{\alpha_j}$, $\alpha_j\ieq\tau_j{-}\mu_j$,
and $\nvec{\beta_j}$, $\beta_j\ieq\tau_{j+1}-\mu_j$, at the endpoints
of the \hbox{$j$-th} segment of the interpolated spiral,
measured with respect to the direction $\nvec{\mu_j}$ of the \hbox{$j$-th} chord. 
To estimate $\alpha_j,\,\beta_j$, the angles ${\xi_j}$ and ${\eta_j}$ are used (\RefFig[(b)]{XiEta}).
Vector $\nvec{\xi_j}$ is the tangent to the circular arc $P_{j-1}P_jP_{j+1}$ at $P_j$;
vector $\nvec{\eta_j}$ is the tangent to the arc $P_jP_{j+1}P_{j+2}$ at~$P_{j+1}$;
both $\xi_j$ and $\eta_j$ are also measured with respect to $\nvec{\mu_j}$:
\Equa{XiEta}{%
    \begin{array}{ll} 
          \sin\xi_j = \dfrac{-c_j\sin\rho_j}{d_{j} }={-}c_jq_{j},\quad\quad&
          \sin\eta_j= \dfrac{c_j\sin\rho_{j+1}}{d_{j+1} }= c_jq_{j+1}\,,\\
          \cos\xi_j =\dfrac{c_{j-1}{+}c_j\cos\rho_j}{d_{j} },\quad&
          \cos\eta_j=\dfrac{c_{j+1}{+}c_j\cos\rho_{j+1}}{d_{j+1} }\,,
    \end{array}\quad   1\le j\le M.\quad
}
Let $\nvec{\vartheta_j}$ denote the tangent to the circular arc $P_{j-1}P_jP_{j+1}$ at $P_j$,
measured in global coordinates.
Note that
\Equa{abij}{
   \aligned
          \tau_j=&\mu_{j-1}{+}\beta_{j-1}=\mu_j+\alpha_j &\So \beta_{j-1}&=\rho_j+\alpha_j, \\
     \vartheta_j=&\mu_{j-1}{+}\eta_{j-1}=\mu_j+\xi_j     &\So  \eta_{j-1}&=\rho_j+\xi_j.
   \endaligned
}

\subsection{Spiral arc in normalized position}

We consider spiral and circular arcs between two neighbouring nodes 
as curves, one-to-one projectable onto the subtending chord. 
Therefore they can be represented as functions $y(x)=f_j(x)$ in the
local coordinate system such that the \hbox{$j$-th} chord becomes the segment $[-c_j,c_j]$ 
of the local $x$-axis:
\Equa{Plotfx}{%
   y(\pm c_j)=0;\quad 
   y'(-c_j)=\tan\alpha_j,\;
   y'(c_j)=\tan\beta_j,\quad
   \kappa_j(x)=\frac{y''}{\left(1+{y'}^2\right)^{3/2}}  
}
is monotone curvature function.

The circular arc, 
resting on the segment $[-c,c]$  of the $x$-axis with the tangent vector 
$\nvec{\varphi}$ at the start point, is denoted as
\equa{%
      \setlength{\unitlength}{1pt}
      \parbox[c]{150pt}{\begin{picture}(150,28)%
         \put(0,0){\Infigw{150pt}{AuxArc}}
      \end{picture}}\quad
   A(x;c,\varphi)=\frac{(c^2-x^2)\sin\varphi}{c\cos\varphi+\sqrt{c^2-x^2\sin^2\varphi}},
   \quad \abs{x}\le c,\quad\abs{\varphi}\le\dfrac{\pi}{2}.
}

The spiral segment has tangent $\nvec{\alpha}$ and curvature $a$ at the startpoint
$A\ieq(-c,0)$, and tangent $\nvec{\beta}$ and curvature~$b$ at the endpoint $B\ieq(c,0)$.
Below two biarcs $AJB$ illustate this data. 
Assuming that the curvature~\eqref{Plotfx} is increasing, $a=\kappa(-c)<\kappa(c)=b$, the following relations are valid:\vspace{1\baselineskip}
\begin{subequations}\label{QVogt}
\begin{align}
    &\alpha+\beta>0\quad\left[\text{or~} \sgn(b{-}a)=\sgn(\alpha{+}\beta)\right];\label{VogtTh} \\[6pt]
    &ac<-\sin\alpha,\quad \sin\beta<bc;\label{acbc} \\[6pt]
    \setlength{\unitlength}{1pt}
    \parbox[b]{170pt}{\begin{picture}(170,1)
         \put(0,0){\Infigw{170pt}{BiarcABv}}
      \end{picture}}\qquad
    &A(x;c,-\beta)< f(x)<  A(x;c,\alpha),\quad \abs{x}\ilt c.\label{LensTh}
\end{align}
\end{subequations}
Eq.\,\eqref{VogtTh} is {\em modified Vogt's theorem}:
(see \cite{AKparab}, St.\,3). 
Eqs.\,\eqref{acbc} are commented in \cite{AKparab}, St.\,6.
Eq.\,\eqref{LensTh} is {\em lens theorem} (see \cite{AKparab}, St.\,5). 
These relations turn to equalities if the curvature is constant.

\section{Bound for a spiral spline}

Interpolation of a curve usually assumes that the neighbouring nodes are located rather closely.
The corresponding restrictions could be expressed in terms of sufficiently small
turning angles~$\rho_j$, \eg, $\abs{\rho_j}\ile\pi/2$.
We apply even more weak constraints,
\Equa{Lim180}{%
    c_{j{-}1}{+}c_j\cos\rho_j\ge 0\mbox{~~~and~~~}
    c_j{+}c_{j{-}1}\cos\rho_j\ge 0.
}
This requires each of two arcs, $P_{j-1}P_j$ and $P_jP_{j+1}$, of the circle $P_{j-1}P_jP_{j+1}$
not to exceed \Deg{180}. 
Three last examples below (red) violate these conditions.\nopagebreak[4]\\
\parbox[c]{\textwidth}{\Infigw{\textwidth}{AuxRho}}\vspace{\baselineskip}

Propositions 1 and 2 constitute Theorem~5 in \cite{AKinterpol}.
\begin{Prop}
The union of lenses, formed by arcs $A(x;c_j,-\eta_j)$ and $A(x;c_j,\xi_j)$
on every chord $P_jP_{j+1}$, $j=1,\ldots,M$, 
covers all possible spirals, matching given
data $\{P_1,\ldots, P_{M+1};\,\tau_1,\tau_{M+1}\}$.
\end{Prop}

\begin{proof}[\bf Proof]
In the below proofs only the case of increasing curvature is considered.
Denote $f_j(x)$ the \hbox{$j$-th} segment of original (or interpolated) spiral.
Inequalities near $f_j(x)$ in the chain
\Equa{Bound1}{%
        A(x;c_j,-\eta_j)\le A(x;c_j,-\beta_j)\le f_j(x)\le A(x;c_j,\alpha_j)\le A(x;c_j,\xi_j)
}
constitute the lens theorem for a spiral arc (Eq.\,\eqref{LensTh} and \RefFig[(a)]{XiEta});
non-strict form ``${\le}$'' takes into account the possibilty of constant curvature.
Outer inequalities express the statement of Prop.\,1 for every \hbox{$j$-th} chord.
To show that  the lens, defined by $(\alpha_j,{-}\beta_j)$, is enclosed by the lens $(\xi_j,{-}\eta_j)$,
we have to prove
\Equa{Enclose}{%
    \xi_j\ge\alpha_j\ge -\beta_j\ge -\eta_j.  
}
The inner inequality, \GE{\alpha{+}\beta}, is Vogt's theorem~\eqref{VogtTh} for increasing (${>}$)
or constant (${=}$) curvature of the spiral arc $f_j(x)$. 
Note also that, by construction,
\Equa{FirstLast}{%
   \alpha_1\ieq\xi_1,\qquad \beta_{M}\ieq\eta_{M}.
}
To illustrate proof of \eqref{Enclose}, the circle of curvature of the spiral at node~$j$ ($j\ieq2$) 
is traced by dashed line in \RefFig[(c)]{XiEta}; 
tangent $\nvec{\alpha_2}$ to the spiral at~$P_2$ is also tangent to this circle.
It is known that
a spiral arc with increasing curvature
intersects the osculating circle from right to left. 
The point $P_{1}$ is located to the right of the circle, and the point $P_{3}$ to the left.
Therefore the circular arc, passing through points $P_{1,2,3}$, 
intersects the osculating circle from right to left,
and $\xi_2\igt\alpha_2$. The case $\xi_2\ieq\alpha_2$ occurs if the segment $P_{1}P_2P_{3}$ 
of the original spiral is coincident with the osculating circle. 
Similarly, from mutual position of the circular arc through $P_{2,3,4}$, 
and the osculating circle at $P_3$ (\RefFig[(d)]{XiEta}) we deduce $\eta_2\ge\beta_2$. 
\end{proof}

The width $\Diam{}$ of the united region is
\equa{
   \Diam{}\ieq\max\limits_{1\le j\le M}\Diam{j},\qquad 
   \Diam{j}=c_j\abs{\tan\frac{\xi_j}2+\tan\frac{\eta_j}2}\approx \frac12 c_j^2\abs{q_j-q_{j+1}}.
}
Increasing $N$ decreases  $c_j$ and $\abs{q_j{-}q_{j+1}}\approx\abs{k_j{-}k_{j+1}}$,
which means cubic convergence of the width~$\Diam{}$ with~$N$.
\RefFig[(c)]{Region}, showing the region for 16 points, illustrates this effect.

\begin{Prop}
If  turning angles $\rho_j$ obey constraints~\eqref{Lim180},
the sequence $q_j$, $j=1,\ldots,N$, of $3$-point curvatures is monotone.
\end{Prop}

\begin{proof}[\bf Proof]
Comparison of~\eqref{Lim180} and~\eqref{XiEta} yields $\abs\xi_j\le \frac{\pi}2$, 
and $\abs\eta_j\le \frac{\pi}2$; so,
\equa{%
   -\eta_j\stackrel{\eqref{Enclose}}\le\xi_j\So
   {-\sin\eta_j} \le  \sin\xi_j
   \quad \stackrel{\eqref{XiEta}}\Longrightarrow\quad
    -q_{j+1} c_j \le -q_jc_j \So
     q_{j+1}\ge q_j. \eqno\qed
}\def\qed{}
\end{proof}



\section{Bound for spline with distinguished vertices (piecewise spiral spline)}

Consider the case when the sequence of 3-point curvatures is not monotone,
although conditions \eqref{Lim180} are satisfied.
This means that the original curve is not a spiral. Assume that 
\begin{subequations}\label{Condab}
\begin{align}
 &\text{\em nodes with minimal/maximal curvature $q_j$ are exactly the vertices of the curve;}\label{Conda}\\[-2pt]
 &\text{\em there is at least one node between two neighbouring vertices.}\label{Condb}
\end{align}
\end{subequations}
Under these assumptions there is also a simple way to construct the bounding region. 
In \RefFig{Ellipse} the ellipse is subdivided by 13 points.
Analyzing the sequence of 3-point curvatures and assuming~\eqref{Condab},
we detect vertices at $P_1,P_5,P_8,P_{10}$.
The 3-point arcs $P_{j-1}P_jP_{j+1}$ are traced only for the triples,
wherein the midpoint $P_j$ is not a vertex (\RefFig[(a)]{Ellipse}).
As before, we obtain the pair of boundary arcs on chords
2--3, 3--4, 6--7, 11--12, 12--13.
We have only one circular arc on each chord, adjacent to a vertex.
Lacking arcs~\eqref{YesNo} are shown in \RefFig[(b)]{Ellipse} by bold (red) lines. 
The angles in~\eqref{YesNo}  are such that, e.\,g., bold arc 1--2 shares tangent at $P_2$ with the arc 2--3--4;
tangent to the arc 2--3--4 at $P_4$ is matched by the 2-point arc 4--5, and so on.

The circle of curvature of the ellipse at the point $P_8$ (\RefFig[(a)]{Ellipse}) shows that
the distances  $\abs{P_7P_8}$, $\abs{P_8P_9}$ are too big,
compared to the local radius of curvature. As consequence, the estimated
bounding regions 7--8 and 8--9 are too wide. 

\begin{Prop}
Under assumptions \textrm{\eqref{Condab}} the pair of boundary arcs for the point set with distinguished vertices is defined as follows:
\settowidth{\tmplength}{$ A(x;c_j,\eta_{j+1}-\rho_{j+1}),\quad$}
\Equa{YesNo}{%
  \begin{array}{ccccl}
    \text{vertex at $P_j$:} && \text{vertex at $P_{j+1}$:} && \text{boundary arcs for the chord $P_jP_{j+1}$:}\\
        \text{\it yes}&&\text{\it no}&& \makebox[\tmplength][l]{$A(x;c_j,-\eta_j),$}  A(x;c_j,-\xi_{j-1}-\rho_j);\\
        \text{\it no}&&\text{\it yes}&& \makebox[\tmplength][l]{$A(x;c_j,\eta_{j+1}-\rho_{j+1}),$} A(x;c_j,\xi_j);\\
        \text{\it no}&&\text{\it no}&&  \makebox[\tmplength][l]{$A(x;c_j,-\eta_j),$} A_j(x;c_j,\xi_j).
   \end{array}        
}
\end{Prop}

\begin{figure}[t]
\centering%
\Pfig{1\textwidth}{Ellipse}{Constructing bounds for curves with distinguished vertices
(dotted line is the original ellipse)}%
\end{figure}


\begin{proof}[\bf Proof]
The case {\em no/no} is copied from \eqref{Bound1}.
Condition \eqref{Condb} excludes the case {\em yes/yes}.

First, we note that for the closed curve there is no need to specify the boundary tangents $\tau_1$,~$\tau_N$: 
to define the turning angles $\rho_1,\,\rho_N$
definitions \eqref{AddPt} should be replaced by $P_0=P_N$ and $P_{N+1}=P_1$, 
assuming $M\ieq N$ (the number of chords is equal to the number of nodes).

Consider the segment (3-4-\underline{5}-6-7) with curvature decreasing in
(3-4-\underline{5}) to the vertex at node~\underline{5} (minimum), 
and then inreasing in (\underline{5}-6-7).
Inequalities \eqref{Enclose} look like
\equa{
       P_4P_5{:}\;\;  \xi_4<\alpha_4<-\beta_4<?\,,\qquad  
       P_5P_6{:}\;\; -\eta_5<-\beta_5<\alpha_5<\text{?`}\,,
}
and yield $-\alpha_5<\eta_5$, and $\beta_4<-\xi_4$. 
Missed right-most restrictions can now be obtained as follows: 
\equa{%
   -\beta_4\Eqref{abij}-\alpha_5-\rho_5<\eta_5-\rho_5;\qquad
   \alpha_5=\beta_4-\rho_5<-\xi_4-\rho_5.
}
Similar reasoning yields similar bounds for the vertex of maximal curvature;
e.\,g., for vertex~8
\equa{
       P_7P_8{:}\;\;  \xi_7>\alpha_7>-\beta_7>\eta_8{-}\rho_8\,;\qquad  
       P_8P_9{:}\;\; -\eta_8<-\beta_8<\alpha_8<-\xi_7{-}\rho_8\,. 
}
Together they look like (with upper signs for a vertex-minimum)
\equa{
    \underbrace{\xi_j\lessgtr \alpha_j\lessgtr -\beta_j\lessgtr \eta_{j+1}-\rho_{j+1}}_%
{\text{chord~}P_jP_{j+1},\text{~vertex at~}P_{j+1}}\,;\qquad  
    \underbrace{-\eta_j\lessgtr -\beta_j\lessgtr \alpha_j\lessgtr -\xi_{j-1}-\rho_j}_%
{\text{chord~}P_jP_{j+1},\text{~vertex at~}P_{j}}\,.  
}
This yields bounds~\eqref{YesNo} for {\em yes/no} and {\em no/yes} cases.
\end{proof}

Note that the discrete curvature plot may look like
\setlength{\unitlength}{1pt}
\begin{picture}(100,14)
\qbezier(0,5)(3,14)(14,14)
\qbezier(14,14)(25,14)(35,14)
\qbezier(35,14)(40,14)(45,7)
\qbezier(45,7)(50,0)(55,0)
\qbezier(55,0)(70,0)(80,0)
\qbezier(80,0)(90,0)(99,7)
\put(0,7){\circle*{3}}
\put(8,13){\circle*{3}}
\put(22,14){\circle*{3}}
\put(33,14){\circle*{3}}
\put(45,7){\circle*{3}}
\put(55,0){\circle*{3}}
\put(65,0){\circle*{3}}
\put(75,0){\circle*{3}}
\put(85,0){\circle*{3}}
\put(99,7){\circle*{3}}
\end{picture},
including the segments of constant curvature, which can be named ``extended vertices''.
Any node of the extended vertex can be chosen as the vertex 
to apply the above described algorithm.

\section{Narrowing the bounding region for spiral splines }


\RefFig{Forced} illustrates the strengthened version
of  Prop.\,1: some of the bounding arcs are replaced by biarcs.
As the consequence, the bound for the whole region
turned into a pair of smooth curves, intersecting at the nodes.

\begin{figure}[t]
\centering%
\Pfig{1\textwidth}{Forced}{Simple bounding region (dashed), and narrowed bounding region (solid lines);
biarc boundaries are shown with an arrow at the joint.}%
\end{figure}

\begin{figure}[t]
\centering%
\Pfig{\textwidth}{Collinear}{Regions for data with 4 cocircular (collinear) points}%
\end{figure}

If 4 points $P_{j-1},\ldots,P_{j+2}$ are cocircular (e.\,g., collinear), we obtain $q_j\ieq q_{j+1}$,
and zero width \EQ{\Diam{j}}. Any spiral, matching such data, 
includes arc $P_{j-1}P_jP_{j+1}P_{j+2}$ of constant curvature~$q_j$.
\quo{Simple} construction, defined by~\eqref{Bound1}, yields the region of zero width
on the segment $P_{j}P_{j+1}$.
Strengthened version~\eqref{Bound2} returns zero width
on the whole segment $P_{j-1}P_jP_{j+1}P_{j+2}$
Example in \RefFig{Collinear} illustrates this situation.\medskip

To describe the narrowed region (Prop.\,5) we first introduce notation for biarc curves.
Because all involved biarcs are one-to-one projectable
onto the chord, they can be considered as functions $B(x;\ldots)$ in the
local coordinate system~\eqref{QVogt}.
The first circular arc $AJ$ of a biarc in~\eqref{QVogt}
has tangent $\nvec{\alpha}$ and curvature $a$ at the startpoint
$A\ieq(-c,0)$, and is smoothly continued at the join point~$J$ by the second arc $JB$ of curvature~$b$
to the endpoint $(c,0)$ with end tangent $\nvec{\beta}$.
The condition of tangency of arcs $AJ$ and $JB$ looks like
\Equa{DefQ}{%
    (ac+\sin\alpha)(bc-\sin\beta)+\sin^2\omega=0,  
    \quad\text{where}\quad \omega=\dfrac{\alpha{+}\beta}{2}
}
is the angular half-width of the associated lens.
Details and reference formulae
for a family of biarcs
are given in \cite{Bilens}.
\begin{figure}[t]
\centering%
\Pfig{\textwidth}{BiarcDeg}{Defining \quo{degenerate} biarcs \eqref{BiarcInf} and \eqref{BiarcZero};
lens boundaries are shown dashed}%
\end{figure}

It is well known that, fitting two-point G$^1$ data (tangents at the endpoints) with a biarc,
we have one degree of freedom: one curvature, either~$a$ or~$b$, can be selected,
the other should be defined from condition of tangency~\eqref{DefQ}.
We use the following notation for biarc curves:
\begin{description}
\item[] $B_1(x;c,\alpha,\beta,a)$ for biarc, matching end tangents $\nvec{\alpha},\,\nvec{\beta}$, and the start curvature~$a$;
\item[] $B_2(x;c,\alpha,\beta,b)$ for a biarc, matching end tangents $\nvec{\alpha},\,\nvec{\beta}$, and the end curvature~$b$;
\item[] $B_0(x;c,\alpha,\beta,p)$, $p\in[0;\infty]$, for a family of {\em short biarcs} \cite{Bilens}
with curvatures
\Equa{abp}{%
\hspace*{-3em}
     a=-\frac1c\left(\sin\alpha - \frac1p \sin\omega\right),\;
     b= \frac1c\left(\sin\beta +  p\sin\omega\right)\;
     \left[p=-\frac{\sin\omega}{ac+\sin\alpha}=\frac{bc-\sin\beta}{\sin\omega}\right].
}
\end{description}

Now define \textit{degenerate biarcs} as follows.
For $\alpha+\beta\gtrless 0$, and either $p=0$ or $p=\infty$
\begin{subequations}\label{BiarcsDeg}
\Equa{BiarcInf}{%
 \aligned%
   &B_0(x;c,\alpha,\beta,0)=
     \lim\limits_{p\to0}B_0(x;c,\alpha,\beta,p)=B_1(x;c,\alpha,\beta,\mp\infty)=A(x;c,-\beta);\\
   &B_0(x;c,\alpha,\beta,\infty)=
     \lim\limits_{p\to{\infty}}B_0(x;c,\alpha,\beta,p)=B_2(x;c,\alpha,\beta,\pm\infty)=A(x;c,\alpha).
  \endaligned
}
For $\alpha+\beta= 0$:
\Equa{BiarcZero}{%
    B_0(x;c,\gamma,-\gamma,p)=
       \lim\limits_{\omega\to0}B_0(x;c,\underbrace{\omega{+}\gamma}_{\alpha}
                                      ,\underbrace{\omega{-}\gamma}_{\beta},p)=A(x;c,\gamma).
}
\end{subequations}
\RefFig[(a)]{BiarcDeg} illustrates definitions \eqref{BiarcInf}:
when, e.\,g., the join point tends to startpoint~$A$, the first curvature~$a$ tends to 
infinite impulse of curvature at $A$, the first subarc of a biarc vanishes;
simultaneously the second arc tends to the lens boundary.
In \RefFig[(b)]{BiarcDeg} the case $\alpha{+}\beta \ieq 2\omega\to 0$ \eqref{BiarcZero} is illustrated
with lenses, narrowing down in width.

\begin{Prop}
Consider a spiral arc $y\ieq f(x)$  
with increasing curvature  $\kappa(x)$.
If boundary angles $\alpha,\beta$ and curvature fall in ranges
\Equa{T4cond}{%
   \alpha'\le\alpha\le\alpha'',\quad  \beta'\le\beta\le\beta'',
   \quad -\infty\le a\le \kappa(x) \le b\le{+\infty},
}
than the curve is bounded by
\Equa{T4}{%
 \aligned
   &\text{if~}\alpha'{+}\beta''\ge0{:}\quad &  B_1(x;c,\alpha',\beta'',a)\le &f(x);\\
   &\text{if~}\alpha''{+}\beta'\ge0{:}\quad & &f(x) \le B_2(x;c,\alpha'',\beta',b).
 \endaligned  
}
\end{Prop}

\begin{figure}[t]
\centering%
\Pfig{.96\textwidth}{ThBBB}{Illustration to Prop.\,4 (shaded region is bilens)}%
\end{figure}

The proof (\cite{AKinterpol}, Th.\,4) includes inequalities
(labels correspond to \RefFig{ThBBB}):
\equa{
  \begin{array}{rl}
     &\overbrace{B_1(x;c,\alpha',\beta''\!,a)}^{AJ_1B}\le
     \overbrace{B_1(x;c,\alpha,\beta''\!,a)}^{AJ_2B}\le
     \overbrace{B_1(x;c,\alpha,\beta,a)}^{AJ_3B}\le 
     f(x);\\[1.5ex]
     f(x)\le\!
	&\underbrace{B_2(x;c,\alpha,\beta,b)}_{AJ_4B}\le
     \underbrace{B_2(x;c,\alpha''\!,\beta,b)}_{AJ_5B}\le
     \underbrace{B_2(x;c,\alpha''\!,\beta',b)}_{AJ_6B}
  \end{array}
}
Inequalities near $f(x)$ result from {\em bilens theorem}
(\cite[Th.\,3]{AKinterpol}, \cite[Th.\,1]{Bilens}).
The rest can be obtained by inspecting the involved circular arcs.

\begin{Prop}
The narrowed bounding region for spiral data 
$\{(x_j,y_j),\:j=1,\ldots,M{+}1,\;\tau_1,\tau_{M+1}\}$
with increasing curvature is defined by
\begin{subequations}\label{Bound2}
\Equa{Bound2Def}{%
   B_0(x;c_j,\alpha'_j,\beta''_j,p'_j)
   \le f_j(x)\le 
   B_0(x;c_j,\alpha''_j,\beta'_j,p''_j).
}
Biarcs' parameters are
\Equa{Bound2Tab}{%
  \begin{tabular}{|c|c|c||c|c|c|}\hline
                 & $j=1$ & $1<j\le M$ &
                 & $1\le j< M$ & $j=M$ \\ \hline
   ${\alpha'_j}_{\strut}=$ & $\xi_1$ &$\max(-\rho_j{-}\xi_{j-1}, -\eta_j)$ &
   ${\beta'_j}_{\strut}=$  & $\max(-\xi_j,\rho_{j+1}{-}\eta_{j+1})$ & $\eta_M$\\ \hline
   ${\alpha''_j}_{\strut}=$ &$\xi_1$ & $\xi_j$ &
   ${\beta''_j}_{\strut}=$ &$\eta_j$ & $\eta_M$ \\ \hline
   $a_j=$  & $-\infty$
            & $ {\dfrac{1}{c_{j-1}}\sin \beta'_{j-1}}^{\strut}_{\strut}$ &
   $b_j=$ & $-{\dfrac{1}{c_{j+1}}\sin\alpha'_{j+1}}_{\strut}^{\strut}$ & $+\infty$ \\ \hline
   $p'_j =$ & \multicolumn{2}{c||}{$-\dfrac{\sin\omega}{a_jc_j+\sin\alpha'_j},\quad\omega={\dfrac{\alpha'_j+\beta''_j}2}^{\strut}$}&
   $p''_j=$ & \multicolumn{2}{c|}{$\dfrac{b_jc_j-\sin\beta'_j}{\sin\omega},\quad\omega={\dfrac{\alpha''_j+\beta'_j}2}_{\strut}$} \\ \hline
  \end{tabular}   
}
\end{subequations}
\end{Prop}


\begin{proof}[\bf Proof] 
Transform~\eqref{Enclose}, following the scheme:
\equa{%
   \begin{array}{l}
     -\xi_{1} \le \beta_{1} \le \eta_{1}\\
     -\eta_2\le \alpha_2 \le \xi_2
   \end{array}
   \So
   \begin{array}{l}
     -\xi_{1}{-}\rho_2 \le\beta_{1}{-}\rho_2 \le \eta_{1}{-}\rho_2\\
     -\eta_2{+}\rho_2\le \alpha_2{+}\rho_2 \le \xi_2{+}\rho_2
   \end{array}
  \So
   \begin{array}{l}
     -\xi_{1}{-}\rho_2 \le \alpha_2 \le \xi_2,\\
     -\eta_2{+}\rho_2\le \beta_{1} \le \eta_{1}.
   \end{array}
}
Combining the first and the last column yields ranges~\eqref{T4cond} for $\beta_1$ and $\alpha_2$:
\equa{%
   \renewcommand{\tmp}[2]{\underbrace{#1}_{\displaystyle#2}}
   \tmp{\max(-\xi_{1},{\rho_2{-}\eta_2})}{\beta'_1} \le \beta_{1} \le \tmp{\eta_{1}}{\beta''_1}\,,\qquad
   \tmp{\max(-\rho_2{-}\xi_{1},{-\eta_2})}{\alpha'_2} \le \alpha_2 \le \tmp{\xi_2}{\alpha''_2}\,.
} 
Similarly, we obtain ranges for pairs $(\beta_2,\alpha_3)$, $(\beta_3,\alpha_4)$, \dots, 
$(\beta_{M-1},\alpha_M)$.
Missing ranges for $\alpha_1$ and $\beta_M$ are known from \eqref{FirstLast}: 
$\alpha'_1\ieq\alpha''_1\ieq\xi_1$,
$\beta'_M\ieq\beta''_M\ieq\eta_M$. 
Two rows of the table \eqref{Bound2Tab} are thus filled.

Now check conditions in~\eqref{T4}. E.\,g., for \GE{\alpha'_j{+}\beta''_j}
we have
\vspace{-.4\baselineskip}%
\equa{%
    \begin{array}{lr}
        j=1{:}\quad& \alpha'_1+\beta''_1 = \xi_1+\eta_1\stackrel{\eqref{Enclose}}{\ge} 0;\\
        j>1{:}& \alpha'_j{+}\beta''_j=\max(-\eta_j, {-}\rho_j{-}\xi_{j-1})+\eta_j=
                         \max(0,\,\eta_j{-}\rho_j{-}\xi_{j-1})\ge 0.
    \end{array}                     
}
Similarly, conditions $\alpha''_j{+}\beta'_j\ige 0$ hold.

To estimate ranges for unknown curvatures $k_j$ in nodes $j\ieq 1,\ldots,M{+}1$,
we make use of inequalities~\eqref{acbc};
non-strict form accounts for the possible case of constant curvature:
\Equa{OnChord}{%
    \begin{array}{rclll}
    \text{on the chord}
    &&P_1P_2{:}\quad&k_1 c_1 \le -\sin\alpha_1,\quad & k_2 c_1\ge\sin\beta_1;\\
    &&P_2P_3{:}     &k_2 c_2 \le -\sin\alpha_2,\quad & k_3 c_2\ge\sin\beta_2;\\
    && \cdots & \cdots & \cdots \\
    &&P_{M-1}P_M{:} &k_{M-1} c_{M-1} \le -\sin\alpha_{M-1},&k_M c_{M-1}\ge\sin\beta_{M-1};\\
    \text{last chord}
    &&P_MP_{M+1}{:}\quad&k_M c_M \le -\sin\alpha_M,  & k_{M+1} c_M\ge\sin\beta_M.
    \end{array}
}
This results in
\equa{%
   \underbrace{%
     a_j=\frac{\sin\beta'_{j-1}}{c_{j-1}}\le
      \frac{\sin\beta_{j-1}}{c_{j-1}}
    }_{\text{except $j=1$}}  
    \le k_j\le 
   \underbrace{%
     \frac{-\sin\alpha_j}{c_j}\le 
     \frac{-\sin\alpha'_j}{c_j}=b_j
    }_{\text{except last node $j=M{+}1$}},  
}
i.\,e.
\ $-\infty\ile\kappa_1(x)\ile b_1$,
\ \ $a_2\ile\kappa_2(x)\ile b_2$,
\ \ \dots,
\ \ $a_M\ile\kappa_M(x)\ile +\infty$.
Prop.\,4 can now be applied to define boundary biarcs:
\Equa{Bound2ab}{%
   B_1(x;c_j,\alpha'_j,\beta''_j,a_j)
   \le f_j(x)\le 
   B_2(x;c_j,\alpha''_j,\beta'_j,b_j).
}
It is easy to verify that degenerate cases, arising in \eqref{Bound2ab}
with $a_1\ieq{-\infty}$, or $b_M\ieq\infty$, or \EQ{\alpha{+}\beta},
result in \quo{simple} bounds~\eqref{Bound1}.
E.\,g.,
if \EQ{\alpha'_j{+}\beta''_j} is the case, $\alpha'_j\ieq{-\beta''_j}\ieq{-\eta_j}$,
and
\equa{%
B_1(x;c_j,\alpha'_j,\beta''_,a_j)=B_0(x;c_j,-\eta_j,\eta_j,p)\Eqref{BiarcZero}A(x;c_j,-\eta_j).
}
Eq.~\eqref{Bound2ab} is rewritten in \eqref{Bound2} in terms of \quo{universal} function $B_0(x;\alpha,\beta,p)$
\eqref{abp}\footnote{%
In~\eqref{Bound2ab} \quo{hidden} degenerate biarcs may occur.
E.\,g., the biarc $B_2(x;c,\alpha'',\beta',b)$, regular at first sight,
becomes degenerate if \EQ{bc-\sin\beta'}, and equal to
$B_1(x;c,\alpha'',\beta',\infty)=B_0(x;c,\alpha'',\beta',0)$.
This happens in the data like shown in \RefFig{Collinear}, chord 2.
}.
\end{proof}

Note also that possible continuation of \eqref{OnChord} is
\equa{%
    q_{j-1}=-\frac{\sin\xi_{j-1}}{c_{j-1}}\le \frac{\sin\beta_{j-1}}{c_{j-1}}\le
     k_j
     \le -\frac{\sin\alpha_j}{c_j}\le \frac{\sin\eta_j}{c_j}=q_{j+1}{:}
}     
{\em unknown curvatures $k_j$ are limited by known 3-point neighbouring curvatures $q_{j\pm1}$}.
The above chosen ranges $[a_j;\,b_j]$ yield more narrow region.


It may also happen that limits $\pm\infty$ for end curvatures~\eqref{Bound2Tab} can be specified from
some additional considerations. E.\,g., in the 3-point data ($M\ieq2$),
shown below, one can assume positivity of curvature, and replace $a_1=-\infty$ by $a_1=0$.
This replaces the lower boundary arc $A(x;c_1,-\eta_1)$ (dashed) by the regular biarc
$B_1(x;c_1,\xi_1,\eta_1,0)$:\nopagebreak[4]\\
\centerline{\Infigw{.8\textwidth}{ForcedK0}}\medskip%

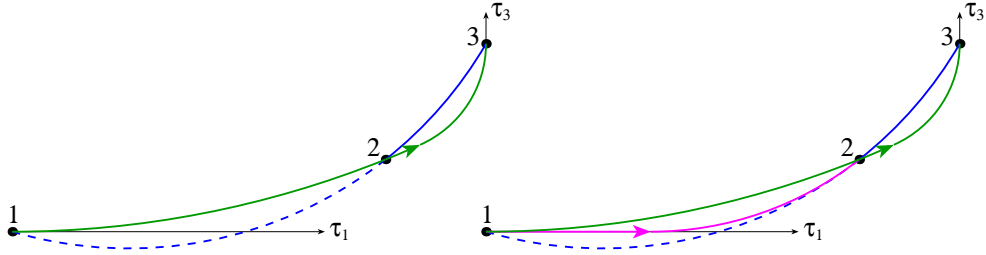
\begin{center}
$\star\star\star$
\end{center}

The specific occasion of the research \cite{AKinterpol} was eventuated in about 1995%
\footnote{%
Article \cite{AKinterpol} was first published in 1998 as 
\href{http://web.ihep.su/library/pubs/prep1998/ps/98-9.pdf}{Preprint IHEP, 98-9, Protvino, 1998}.
}, 
in the workshop of the Institute for High Energy Physics (IHEP, Protvino).
Research originated from a practical problem of inspection 
of a certain cam profile: there was, on a small part of the \Deg{360}-profile, a big discrepancy between
nominal and measured curves. The reason, expressed in present-day terms,
was that the design of the profile included very unfair point subsets.

While exploring the situation, the author has become interested to find
some quantitative measure of such drawbacks in drawings,
possibly, a tool for an expert-metrologist. 
Purely geometrically the problem was 
shaped
as follows: to describe a set of curves {\em with as few vertices as possible}, 
matching given interpolation data. 
The simplest case
(spiral, zero vertices) gave a well restricted solution,
justifying such problem setting.
Proposition\,3 is the step towards non-spiral curves.

\appendix
\def\appendixname{Appendix}

\section{Application to computer-aided tolerancing}\label{sec:Test}

\subsection{Bounding regions in the view of tolerance control}

Bounding region is not quite a new concept for industrial design.
Returning to \RefFig{Tolerance}, consider the hole, dimensioned as $\Diam{}\,26{\pm}0.2$.
This can be treated as the bounding region (in the form of ring) for the circular profile.
The given tolerance admits inaccuracy of order~$0.2$ in coordinates of the center of the hole.
If declared as the datum element, the hole defines the origin of the coordinate system;
consequently, any other dimension of the type \quo{coordinate} on this drawing should be agreed with this inaccuracy.
E.\,g., setting tolerance of 
order~$0.1$ for the dimension $100$ would be an evident error in the drawing.

Normally such error will be corrected by an expert-metrologist before starting the manufacturing process.
Possible correction could be specifying {\em the roundness tolerance~$\varepsilon\ll 0.1$} for the hole.
This allows the value of diameter still to be in the range $25.8\ile\Diam{}\ile26.2$,
but reduces the width of the ring-shaped bounding region to much smaller acceptable value.

Note that, given all dimensions and tolerances,
one can usually construct the exact model of the ideal part, corresponging, for the case of the hole,
to $\Diam{}\,26{\pm}0$.
The position of any point of the circular and straight line elements will be perfectly determined.
But this is not so
for the  elements like profile $ABCD$.
Since it is not presented by exact equations,
and the interpolation method is not exactly specified,
there occurs certain {\em nondeterminancy} is constructing the nominal profile itself.

Assuming a coordinate inspection machine as a device for tolerance control,
one should assure precision of coordinate measurements about 10\% of tolerances under inspection.
The more accurate is the device, the more precisely the distance from the measured point to a circle/line
can be defined. Again, this is not the case for the curvilinear profile: 
one could define exact distance, say, to the currently chosen interpolant. 
The latter could be different from one, previously used in NC-machining process.

So, the nondeterminancy of the profile behaves as inevitable addition to the measurement error.
Constructing the bounding region
estimates this nondeterminancy, 
and allows us to set its acceptable limits, called by the accuracy requirements.
The width of the region is the objective measure of fairness of the given point set,
and may serve as the criterion to demand a designer for more detailed profile description.

For traditional interpolation methods
nothing ensures that some interpolant passes within the prescribed bounding region.
But, on the other hand, one can verify whether it does or not.
An example is discussed below.

\begin{figure}[t]
\centering%
\Pfig{.96\textwidth}{Cubic}{Interpolating data with a cubic spline
}%
\end{figure}

\subsection{Test with cubic spline interpolation}

Algorithms of interpolation with spirals are not widely distributed with CAD software.
On the other hand, interpolating spiral data by, e.\,g., a cubic spline,
will usually not return a spiral.
However, as soon as the region is known, wherein the sought for curve definitely lies,
acceptance test for any interpolation method can be performed.
We add to this note an example of such test.
In \RefFig{Cubic} given data points and end tangents
were interpolated by a cubic spline curve. 
The nodes in the example are spaced out such as one could visually estimate
if the spline fits the prescribed region
(three last chords are shown magnified).

The spline can be parametrized by accumulated chord length,
which approximates arc length of the interpolated curve.
Assuming that the 3-rd order Bezi\'er polynomials $x(t),\,y(t)$ are close to the Taylor series  $x(s),\,y(s)$
of the naturally parametrized unknown curve
gives rise to set boundary conditions at the start point of the cubic spline as
$(x'(0),y'(0))\approx\nvec{\tau_1}$, and similarly at the endpoint.

Points setting of 5 first nodes, compared to expected curvatures,
is evidently unacceptable for good interpolation.
It is exhibited in both location of the spline out of prescribed region,
and large beatings in the curvature plot.

Situation on the right side looks much better,
although points are set still too rare,
compared with typical industrial drawings.
Recall that the bounding region encloses some unknown,
but well defined curve
(\eg, solution of a differential equation, or some other implementation of designer's concept).
With such test it is possible:
\begin{itemize}\setlength{\parskip}{0pt}\setlength{\topsep}{0pt}\setlength{\itemsep}{2pt}%
\item
to estimate whether the bounding region is sufficiently narrow, i.\,e. the unknown curve
is well/ill-defined by the given point set;
\item
to estimate closeness of the unknown curve and the interpolant,
which, even for non-spiral interpolation method, may fall inside the bounding region.
\end{itemize}

\subsection{Is the interpolant in \RefFig{Cubic} \quo{unfair}?}

Appearance of the curvature plot in \RefFig{Cubic},
with curvature extrema in every node and between them,
is rather similar to the examples in \hbox{\cite[Figures 9{.}8--9{.}13]{Farin}}.
The question arises if the non-monotonicity of the curvature plot is 
essential, and should be somehow corrected.

The general answer depends on the values of curvatures,
and the particular reason to avoid the beatings.
Analyzing the problem, one should not forget 
that the objects of computer-aided design become the objects of further manufacturing,
which is not geometrically precise.
We can compare the model and the real part (cam profile, highway path, etc.),
and find that the latter fits well prescribed tolerances.
But small difference between two curves could result in 
rather big distortions of curvature plot:\\
\Infigw{1\textwidth}{RoundXY}\\

In the above illustration the data is chosen on the circular arc of radius 10
with $P_1\ieq(0,0)$, and \EQ{\tau_1}.
Ideal descrete curvature plot is the line $q(i)=0{.}10=\const$.
Manufacturing/measurement errors were imitated by rounding the coordinates
to two decimal digits:
$P_2\approx(0{.}5233,\,0.0137)$ was repalced by $(0.52,\,0.01)$,
$P_3\approx(1{.}0453,\,0.0548)$              by $(1.05,\,0.05)$,
\dots, and
$P_{21}\approx(8{.}6603,\,5.0)$           by $(8.66,\,5.0)$.
This rounding caused well visible beatings in discrete curvature plot.
The geometric reason of such behavior is evident:
for a small circular arc $ABC$ a slight shift of the midpoint~$B$ 
effects a dramatic change in radius or curvature.
The smaller is the angular measure of the arc, the greater is this effect.\smallskip

The above considerations seem to agree with 
the note in \cite[Sec.\,23.1]{Farin}: the definition of fairness,
based on minimum of curvature extrama,
{\em \quo{is certainly subjective; however, it has proved to be a practical concept}.}



\begin{thebibliography}{9}           
%
\bibitem{AKinterpol}
Kurnosenko~A.I. 
{\em Interpolational properties of planar spiral curves.}
Fund. and Applied Math. {7}(2001), N2, 441--463 (in Russian). 
%
\bibitem{Farin}
Farin,\,G.
{\em Curves and Surfaces for CAGD: A Practical Guide}, 5-th edition,
Acad. Press, 2002.
%
\bibitem{AKparab}
Kurnosenko~A.I.
{\em
Applying inversion to construct planar, rational, spiral curves that satisfy two-point $G^2$ 
Hermite data.}
Comp. Aided Geom. Design 27(2010), 262--280.
%
\bibitem{Bilens}
Kurnosenko~A.I.
{\em Biarcs and bilens.}
Computer Aided Geometric Design, 30 (2013), 310--330.
%
\newcommand{\POMI}{Zapiski nauch. sem. \hbox{POMI}}
\Skip{%
\bibitem{Pomi}
Kurnosenko~A.I.
{\em General properties of spiral plane curves.}
Journal of Math. Sciences, {161}, No.\,3(2009), 405-418;
(translated from \POMI, {\bf 353}(2008), 93--115).

Kurnosenko~A.I.
{\em Short spirals.}
Journal of Math. Sciences, {175}, No.\,5(2011), 517--522;
(translated from \POMI, {\bf 372}(2009), 34--43).

Kurnosenko~A.I.
{\em Long spirals.}
Journal of Math. Sciences, {175}, No.\,5(2011), 523--527;
(translated from \POMI, {\bf 372}(2009), 44--52).
}
\end{thebibliography}
\end{document}